\documentclass{amsart}
\newtheorem{theorem}{Theorem}[section]
\newtheorem{lemma}[theorem]{Lemma}

\newtheorem{es}[theorem]{Example}
\newtheorem{remark}[theorem]{Remark}
\def\erre{{\rm I\!R}}
\def\RR{{\rm I\!R}}

\def\enne{{\rm I\!N}}

\def\ZZ{{\mathbb Z} }

\title[Existence of
solutions...]{Existence of
solutions for $p$-Laplacian discrete equations}
\author{Giovanni Molica Bisci}
\address[G. Molica Bisci]{Dipartimento P.A.U., Universit\`a  degli
Studi Mediterranea di Reggio Calabria, Salita Melissari - Feo di
Vito, 89124 Reggio Calabria, Italy} \email{gmolica@unirc.it}
\author{Du\v{s}an Repov\v{s}}
\address[D. Repov\v{s}]{Faculty of Education, and Faculty of Mathematics and Physics\\ University of Ljubljana, POB 2964, Ljubljana, Slovenia 1001}
\email{dusan.repovs@guest.arnes.si}

\thanks{Typeset by \LaTeX}
\begin{document}
\begin{abstract}
 This work is devoted to the study of the existence of at least one (non-zero) solution to a problem involving the discrete $p$-Laplacian. As a special case, we derive an
existence theorem for a second-order discrete problem, depending on a positive real parameter $\alpha$, whose prototype is given by
$$
\left\{
\begin{array}{l}
-\Delta^2u({k-1})=\alpha f(k,u(k)),\quad \forall\; k  \in
{\mathbb{Z}}[1,T] \\
{u(0)=u({T+1})=0.}\\
\end{array}
\right.
$$
Our approach is based on variational methods in finite-dimensional setting.
\end{abstract}
\maketitle

\section{Introduction}
We are interested in investigating nonlinear
discrete boundary value problems by using variational methods. This
approach
has been recently adopted, for instance, in \cite{APR1,APR2,BX} and \cite{JiangZhou,LiangWeng,zhangzhangliu}.\par
More precisely, for every $a,b\in \ZZ$, such that $a<b$, set
$\ZZ[a,b]:=\{a,a+1,...,b\}$ and let $T\geq 2$ be a positive integer.\par
 The aim of this paper is to prove some existence results for the following discrete problem
\begin{equation}\tag{$D_{f}$} \label{dani3}
\left\{
\begin{array}{l}
-\Delta(\phi_p(\Delta u({k-1})))
= f(k,u(k)) ,\quad \forall\; k \in
\ZZ[1,T]
\\ {u(0)=u({T+1})=0, } \quad \quad \qquad \qquad \qquad \\
\end{array}
\right.
\end{equation}
where $p>1$, $\phi_p:\erre\rightarrow \erre$ is given by $\phi_p(s):=|s|^{p-2}s$, for every $s\in\erre$, $f:\ZZ[1,T]\times\erre\rightarrow \erre$ is a continuous function, and
$\Delta u(k-1):=u(k)-u(k-1)$ is the
forward difference operator. \par

\indent In recent years equations involving the discrete $p$-Laplacian operator, subject to different boundary conditions, have been widely studied by many authors and several approaches.

In particular, Problem \eqref{dani3} has been previously studied, for instance, in \cite{APR2,Chu,JiangZhou2} by using various methods. See the recent papers \cite{GG, MR} for the discrete anisotropic case.\par
 Motivations for this interest arose in by different fields of
research, such as computer science, mechanical engineering, control systems, artificial or biological neural networks,
economics and others.\par
 Moreover, the main background in the real world for the discrete $p$-Laplacian operator are the problems on the boundary between different substances.\par
Set
$$
c(p,T):=\left\{
\begin{array}{ll}
\displaystyle\frac{1}{p}\left[\displaystyle\left(\frac{2}{T}\right)^{p-1}+\left(\frac{2}{T+2}\right)^{p-1}\right] & {\rm if}\,\,\, T\,\, {\rm is\,\,even} \\
 &\\
\displaystyle \frac{2^p}{p(T+1)^{p-1}} & {\rm if}\,\,\, T\,\, {\rm is\,\,odd.} \\
\end{array}
\right.
$$

 Via variational approach, we are able to prove the existence of a
solution for problem \eqref{dani3} by requiring that
$$
\frac{\displaystyle\sum_{k=1}^{T}\max_{|\xi|\leq \varepsilon} \int_0^{\xi}f(k,s)ds}{\varepsilon^p}<c(p,T),
$$
for some $\varepsilon>0$. See condition \eqref{condizione} in Theorem \ref{Esistenza}.\par
Next, by using Theorem \ref{Esistenza}, we study a parametric version of problem \eqref{dani3}, defined as follows
\begin{equation}\tag{$D^{f}_\alpha$} \label{dani34}
\left\{
\begin{array}{l}
-\Delta(\phi_p(\Delta u({k-1})))
= \alpha f(k,u(k)) ,\quad \forall\; k \in
\ZZ[1,T]
\\ {u(0)=u({T+1})=0, } \quad \quad \qquad \qquad \qquad \\
\end{array}
\right.
\end{equation}
 where $\alpha$ is a positive real parameter.\par In this case, requiring a suitable behaviour of the potentials at zero and at infinity, we obtain, for sufficiently large $\alpha$, the existence of at least one positive solution for problem \eqref{dani34}, see Theorem \ref{Esistenza2}. This result can be achieved exploiting Theorem \ref{Esistenza} together with the well-known variational characterization of the first eigenvalue of the $p$-Laplacian operator in the finite-dimensional context (see \cite{APR2}).\par
 The simplest example we can deal with is
a second-order boundary value problem.\par

\begin{theorem}\label{Esistenza3}
Let $f:\ZZ[1,T]\times[0,+\infty)\rightarrow [0,+\infty)$ be a continuous function satisfying the following hypotheses$:$
$$
\lim_{\xi\rightarrow +\infty}\frac{\displaystyle\sum_{k=1}^{T}\int_0^{\xi}f(k,t)dt}{\xi^2}=0
$$
and
$$
\gamma_\kappa:=\liminf_{\xi\rightarrow 0^+}\frac{\displaystyle\int_0^{\xi}f(k,t)dt}{\xi^2}>0,
$$
for every $k\in \ZZ[1,T]$.
Then for every
$$
\alpha>\frac{2}{\displaystyle\min_{k\in \ZZ[1,T]}\gamma_k}\sin^{2}\left(\frac{\pi}{2(T+1)}\right),
$$
the following second-order discrete problem
\begin{equation} \tag{$S_\alpha$} \label{dxi}
\left\{
\begin{array}{l}
-\Delta^2u({k-1})=\alpha f(k,u(k)),\quad \forall\; k  \in
{\mathbb{Z}}[1,T] \\
{u(0)=u({T+1})=0,}\\
\end{array}
\right.
\end{equation}
has at least one positive solution.
\end{theorem}

 We remark that the results obtained for second-order discrete equations in \cite{BX, Henderson} and our theorems are mutually independent. Moreover, the approach adopted here can be used studying the discrete counterpart of the following problem
\[
\left\{
\begin{array}{l}
\displaystyle\frac{\partial^2 u}{\partial x^2}+\frac{\partial^2 u}{\partial y^2}+f((x,y),u(x,y))=0,\\
\smallskip
        u(x,0)=u(x,n+1)=0,\,\,\,\forall x\in (0,m+1)\\
        \smallskip
        u(0,y)=u(m+1,y)=0,\,\,\,\forall y\in (0,n+1),
\end{array}
 \right.
  \]
  where $m,n\in \enne\setminus\{0\}$ and $f$ is a suitable continuous function. See \cite{GO} and \cite{IM} for details. We refer to the monograph of Cheng \cite{Cheng} for a geometrical interpretation of this equations.\par
The plan of the paper is as follows. Section 2 is devoted to our abstract framework and preliminaries. Successively, in Section 3 we prove our main result (see Theorem \ref{Esistenza}). The parametric case is discussed in the last section (see Theorem \ref{Esistenza2}), where, a concrete example of an application is also presented (see Example \ref{esempio0}).

\section{Abstract Framework}\label{section2}

\indent On the $T$-dimensional Banach space
 $$
 H:=\{u:\ZZ[0,T+1]\rightarrow\RR: u(0)=u({T+1})=0\},
 $$
 endowed by the norm
$$
\|u\|:=\left(\sum_{k=1}^{T+1}|\Delta u({k-1})|^p\right)^{1/p},
$$
we define the functional $J:H\rightarrow \RR$ given by
$$
J(u):=\frac{1}{p}\sum_{k=1}^{T+1}|\Delta u({k-1})|^p-\sum_{k=1}^{T}\displaystyle\int_0^{u(k)} f(k,t)dt,
$$
for every $u\in H$.\par
 We recall that a \textit{solution} of problem \eqref{dani3} is a function $u\in H$ such that
$$
\begin{array}{l} {\displaystyle \sum_{k=1}^{T+1}\phi_p(\Delta u({k-1}))\Delta v(k-1)}\\
\displaystyle \qquad\qquad\qquad\qquad=\sum_{k=1}^{T}f(k,u(k))v(k),\,\,\,\,\,\,\,
\end{array}
$$
for every $v \in H$.\par
 We observe that
problem~\eqref{dani3} has a variational structure. Indeed, the functional $J$ is differentiable in $u\in H$ and one has
$$
\langle J'(u), v\rangle = {\displaystyle \sum_{k=1}^{T+1}\phi_p(\Delta u({k-1}))\Delta v(k-1)}
$$
$$
\qquad \qquad \qquad \qquad\, -\sum_{k=1}^{T}f(k,u(k))v(k),
$$
for every $v \in H$.\par
 Thus critical points of $J$ are solutions to problem~\eqref{dani3}.
In order to find these critical points, we will make use of the following local minimum result due to Ricceri (see \cite{ricceri}) recalled here on the finite-dimensional setting.

\begin{theorem}\label{CV}
 Let $(E,\|\cdot\|)$ be a finite-dimensional Banach space and let $\Phi,\Psi:X\rightarrow \erre$ be two lower semicontinuous functionals, with $\Psi$ coercive and $\Phi(0_E)=\Psi(0_E)=0$. Further, set $$J_\mu:=\mu\Psi+\Phi.$$ Then for each $\sigma>\displaystyle\inf_{u\in X}\Psi(u)$ and each $\mu$ satisfying
 $$
\mu>-\frac{\displaystyle\inf_{u\in \Psi^{-1}((-\infty,\sigma])}\Phi(u)}{\sigma}
 $$
 the restriction of $J_\mu$ to $\Psi^{-1}((-\infty,\sigma))$ has a global minimum.
 \end{theorem}
\indent See \cite{R2,R3,R4} for related abstract critical points results. We also mention the monograph \cite{k2} for some topics on variational methods adopted in this paper and \cite{A} for general facts on finite difference equations.
\section{The Main Result}

By \cite[Lemma 4]{cit} one has that
\begin{equation}\label{immersione2}
\|u\|_\infty:=\displaystyle\max_{k\in
\ZZ[1,T]}|u(k)|\leq \frac{1}{\kappa}\|u\|,
\end{equation}
for every $u\in H$, where
\begin{equation*}
\kappa:=\left\{
\begin{array}{ll}
\left[\displaystyle\left(\frac{2}{T}\right)^{p-1}+\left(\frac{2}{T+2}\right)^{p-1}\right]^{1/p} & {\rm if}\,\,\, T\,\, {\rm is\,\,even} \\
 &\\
\displaystyle \frac{2}{(T+1)^{(p-1)/p}} & {\rm if}\,\,\, T\,\, {\rm is\,\,odd.} \\
\end{array}
\right.
\end{equation*}

\begin{remark}\rm{
 Note that
$$
\left[\left(\frac{2}{T}\right)^{p-1}+\left(\frac{2}{T+2}\right)^{p-1}\right]^{-1/p}<\frac{(T+1)^{(p-1)/p}}{2}.
$$
\noindent Indeed, since the continuous function $\theta:(0,T+1)\rightarrow \erre$ defined by
$$
\theta(s):=\frac{1}{(T-s+1)^{p-1}}+\frac{1}{s^{p-1}},
$$
attains its minimum $\displaystyle\frac{2^p}{(T+1)^{p-1}}$ at $\displaystyle s=\frac{T+1}{2}$, one has
$$
\displaystyle\frac{2^p}{(T+1)^{p-1}}< \theta(T/2).
$$
\noindent Then
$$
\frac{2}{(T+1)^{(p-1)/p}}< \left[\displaystyle\left(\frac{2}{T}\right)^{p-1}+\left(\frac{2}{T+2}\right)^{p-1}\right]^{1/p}=\theta(T/2)^{1/p},
$$
\noindent and the conclusion is achieved.}
\end{remark}
Set
$$
F_k(\xi):=\int_0^{\xi}f(k,s)ds,
$$
for every $k\in \ZZ[1,T]$ and $\xi\in\RR$.\par
\smallskip
With the above notations our result reads as follows.

\begin{theorem}\label{Esistenza}
Let $f:\ZZ[1,T]\times\erre\rightarrow \erre$ be a continuous function and assume that there exists $\varepsilon>0$ such that
\begin{equation}\label{condizione}
\frac{\displaystyle\sum_{k=1}^{T}\max_{|\xi|\leq \varepsilon} F_k(\xi)}{\varepsilon^p}<\frac{\kappa^p}{p}.
\end{equation}
Then problem \eqref{dani3} has at least one solution such that $\|u\|_\infty<\varepsilon$.
\end{theorem}
\begin{proof}
Let us apply Theorem \ref{CV} by choosing $E:=H$, and
$$
 \Phi(u):=-\displaystyle\sum_{k=1}^{T} F_k(u(k)),\qquad \Psi(u):=\|u\|^p,
$$
for every $u\in E$.\par
\indent Taking $\sigma:=\kappa^p\varepsilon^p$, clearly $\sigma>\displaystyle\inf_{u\in E}\Psi(u)$. Moreover, let us estimate from the above the following quantity
$$
\varphi(\sigma):=\frac{\displaystyle \sup_{u\in\Psi^{-1}((-\infty,\sigma])}\Phi(u)}{\sigma}.
$$

\indent Inequality \eqref{immersione2} yields
$$
\Psi^{-1}((-\infty,\sigma])\subseteq \left\{u\in E: \|u\|_{\infty} \le \varepsilon\right\}.
$$
Thus one has that
$$
\varphi(\sigma)\leq \frac{\displaystyle\sum_{k=1}^{T}\max_{|\xi|\leq \varepsilon} F_k(\xi)}{\kappa^p\varepsilon^p}.
$$
\indent Hence it follows, by \eqref{condizione}, that
$$
\frac{\displaystyle\sup_{u\in \Psi^{-1}((-\infty,\sigma])}\displaystyle\sum_{k=1}^{T} F_k(u(k))}{\sigma}<\frac{1}{p},
$$
that is,
$$
\frac{1}{p}>-\frac{\displaystyle\inf_{u\in \Psi^{-1}((-\infty,\sigma])}\Phi(u)}{\sigma}.
$$
 \indent
    Therefore, the assertion of Theorem \ref{CV} follows and the existence of one solution $u\in \Psi^{-1}((-\infty,\sigma))$ to our problem is established.
\end{proof}
\begin{remark}\label{Y1}\rm{If in Theorem \ref{Esistenza} the function $f$ is nonnegative, hypothesis \eqref{condizione} assumes a simpler form
$$
\frac{\displaystyle\sum_{k=1}^{T} F_k(\varepsilon)}{\varepsilon^p}<\frac{\kappa^p}{p}.
$$
Moreover, if for some $\bar k\in\ZZ[1,T]$, $f(\bar k,0)\neq 0$, the obtained solution is clearly non-zero.
}
\end{remark}

\section{A parametric case}

In this section we shall study the following discrete parametric problem
\begin{equation}\tag{$D^{f}_\alpha$} \label{dani34}
\left\{
\begin{array}{l}
-\Delta(\phi_p(\Delta u({k-1})))
= \alpha f(k,u(k)) ,\quad \forall\; k \in
\ZZ[1,T]
\\ {u(0)=u({T+1})=0, } \quad \quad \qquad \qquad \qquad \\
\end{array}
\right.
\end{equation}

\noindent where $\alpha$ is a real positive parameter.\par

\indent For our goal, in order to obtain positive solutions to problem \eqref{dani34}, i.e. $u(k)>0$ for each $k\in \ZZ[1,T]$, we
shall need the following consequence of the {strong comparison
principle}, see \cite[Lemma 2.3]{APR2}.
\begin{lemma}\label{lemma}
If
\[
-\Delta(\phi_p(\Delta u(k-1)))\geq 0,\quad \forall\; k\in \ZZ[1,T]
\]
\[
u(0)\geq 0, \quad  u(T+1)\geq 0,
\]
then either $u>0$ in $\ZZ[1,T],$ or $u\equiv 0$.
\end{lemma}

Moreover, let $\lambda_{1,p},$ $\varphi_1>0$ be the first eigenvalue and eigenfunction of the problem
\begin{equation}\tag{$D_{\lambda, p}$} \label{dani35}
\left\{
\begin{array}{l}
-\Delta(\phi_p(\Delta u({k-1})))
= \lambda \phi_p(u({k})) ,\quad \forall\; k \in
\ZZ[1,T]
\\ {u(0)=u({T+1})=0. } \quad \quad \qquad \qquad \qquad \\
\end{array}
\right.
\end{equation}
\indent As observed in \cite{APR2}, the following variational characterization
\begin{equation}\label{lambda1}
\lambda_{1, p}=\min_{E\setminus\{0_H\}}\frac{\displaystyle\sum_{k=1}^{T+1}|\Delta u({k-1})|^p}{\displaystyle\sum_{k=1}^{T}|u(k)|^p},
\end{equation}
holds.\par
\indent Taking into account the above facts, an important consequence of Theorem \ref{Esistenza} is the following.

\begin{theorem}\label{Esistenza2}
Let $f:\ZZ[1,T]\times[0,+\infty)\rightarrow [0,+\infty)$ be a continuous function satisfying the following hypotheses$:$
$$
\lim_{\xi\rightarrow +\infty}\frac{\displaystyle\sum_{k=1}^{T}F_k(\xi)}{\xi^p}=0,
$$
and
$$
\gamma_k:=\liminf_{\xi\rightarrow 0^+}\frac{\displaystyle F_k(\xi)}{\xi^p}>0,
$$
for every $k\in \ZZ[1,T]$. Then for every
$$
\alpha>\frac{\lambda_{1, p}}{p\displaystyle\min_{k\in \ZZ[1,T]}\gamma_k},
$$
problem \eqref{dani34} has at least one positive solution.
\end{theorem}

\begin{proof} Let $\alpha$ be as in the conclusion, and define
$$
\widetilde{f}(k,t):=
\left\{\begin{array}{ll}
f(k,t) & \mbox{if}\,\,\,t\geq 0\\
f(k,0) & \mbox{if}\,\,\,t<0.
\end{array}\right.
$$
for every $k \in
\ZZ[1,T]$. Consider now the following problem
\begin{equation}\tag{$D^{\widetilde{f}}_\alpha$} \label{dani36}
\left\{
\begin{array}{l}
-\Delta(\phi_p(\Delta u({k-1})))
= \alpha \widetilde{f}(k,u(k)) ,\quad \forall\; k \in
\ZZ[1,T]
\\ {u(0)=u({T+1})=0. } \quad \quad \qquad \qquad \qquad \\
\end{array}
\right.
\end{equation}
\indent By Lemma \ref{lemma}, every non-zero solution of problem \eqref{dani36} is positive. Furthermore, every positive solution of \eqref{dani36} also solves our initial problem \eqref{dani34}. Now, since
$$
\lim_{\xi\rightarrow +\infty}\frac{\displaystyle\sum_{k=1}^{T}F_k(\xi)}{\xi^p}=0,
$$
there exists $\varepsilon>0$ such that
$$
\frac{\displaystyle\sum_{k=1}^{T} F_k(\varepsilon)}{\varepsilon^p}<\frac{\kappa^p}{p}.$$
\indent Hence, bearing in mind Remark \ref{Y1}, condition \eqref{condizione} of Theorem \ref{Esistenza} holds.\par
  Thus problem \eqref{dani34} admits a solution $u_\alpha\in H$ with $\|u_\alpha\|<\varepsilon$. In conclusion, we shall prove that $0_{H}$ is not a local minimum of
the functional
$$
J_\alpha(u):=\frac{1}{p}\sum_{k=1}^{T+1}|\Delta u({k-1})|^p-\alpha\sum_{k=1}^{T}\displaystyle\int_0^{u(k)} f(k,t)dt,
$$
i.e. the obtained solution $u_\alpha$ is non-zero.\par
For this purpose, let us observe that the first eigenfunction $\varphi_1\in H$ is positive and it follows by (\ref{lambda1}) that
\begin{equation}\label{ScV}
\|\varphi_1\|^p=\lambda_{1, p}\sum_{k=1}^{T}\varphi_1(k)^p.
\end{equation}
\indent Since
$$
\gamma_k>\displaystyle\min_{k\in \ZZ[1,T]}\gamma_k>\frac{\lambda_{1, p}}{p\alpha},
$$
for every $k\in \ZZ[1,T]$,
 there exists $\delta>0$ such that
\begin{equation}\label{ScV2}
F_k(\xi)>\frac{\lambda_{1, p}}{p\alpha}\xi^p,
\end{equation}
\noindent for every $k\in \ZZ[1,T]$ and $\xi\in (0,\delta)$.\par
 Hence, we can define $\theta_{\zeta}(k):=\zeta \varphi_1(k)$, for every $k\in \ZZ[0,T+1]$, where $$\zeta\in \Lambda_\delta:=\left(0,\displaystyle\frac{\delta}{\displaystyle\max_{k\in \ZZ[1,T]}\varphi_1(k)}\right).$$
\indent Taking into account \eqref{ScV2} and \eqref{ScV}, we easily get
 \begin{eqnarray*}
\alpha\sum_{k=1}^{T}F_k(\theta_\zeta(k)) &> & \frac{\lambda_{1, p}\displaystyle}{p} \sum_{k=1}^{T}\theta_\zeta(k)^{p}\\\nonumber
               &=& \frac{1}{p}\|\theta_\zeta\|^p,\nonumber
\end{eqnarray*}
\noindent that is,
$$
J_\alpha(\theta_\zeta)=\frac{1}{p}\|\theta_\zeta\|^p-\alpha\sum_{k=1}^{T}F_k(\theta_\zeta(k))<0,
$$
\noindent for every $\zeta\in \Lambda_\delta$. The proof is thus complete.
\end{proof}

\begin{remark}\rm{In Theorem \ref{Esistenza2}, looking at the behaviour of the function
$$
h(\xi):=\frac{\displaystyle\sum_{k=1}^{T}F_k(\xi)}{\xi^p},\,\,\,\,\,\,(\forall\,\xi >0)
$$
at infinity, the existence of one positive solution has been proved.
On the other hand, if the function $f(k,\cdot):\erre\rightarrow \erre$ has a $s$-sublinear potential $F_k$ with $s<p$, for every $k\in \ZZ[1,T]$, the behaviour at zero of the map
$$
\chi(\varepsilon):=\frac{\displaystyle\sum_{k=1}^{T}\max_{|\xi|\leq \varepsilon} F_k(\xi)}{\varepsilon^p},\,\,\,\,\,\,(\forall\,\varepsilon >0)
$$
influences the existence of multiple solutions. More precisely, requiring that
\begin{equation}\label{ScVG}
\lim_{\varepsilon\rightarrow 0^+}\chi(\varepsilon)=0,
\end{equation}
by using variational arguments, one can prove that there exists
a real interval of parameters $\Lambda$ such that, for every $\alpha\in \Lambda$, the problem \eqref{dani34} admits at least three solutions. If, instead of \eqref{ScVG}, we assume that
\begin{equation}\label{ScV4}
\chi(c)<\frac{2^{p-1}}{(T+1)^{p-1}}\left(h(d)-\frac{c^p}{d^p}\chi(c)\right),
\end{equation}
\noindent for some positive constants $c<d$, then for every
$$
\alpha\in\left]\frac{2}{p\displaystyle\left(h(d)-\frac{c^p}{d^p}\chi(c)\right)},\frac{2^p}{p\chi(c)(T+1)^{p-1}}\right[,
$$
there exist at least three distinct solutions of the problem \eqref{dani34}.
Clearly condition \eqref{ScV4} is technical and quite involved. Finally, we also note that a more precise result can be obtained if $T$ is even.
}
\end{remark}

\begin{remark}\rm{
It is easy to see that Theorem \ref{Esistenza3} in Introduction is a consequence of Theorem \ref{Esistenza2} bearing in mind that the first eigenvalue of the problem
\begin{equation} \tag{$D_\lambda$} \label{Do}
\left\{
\begin{array}{l}
-\Delta^2u({k-1})=\lambda u(k),\quad \forall\; k  \in
\ZZ[1,T] \\
{u(0)=u({T+1})=0,}\\
\end{array}
\right.
\end{equation}
is given by
$$
\lambda_1:=4\sin^2\left(\frac{\pi}{2(T+1)}\right),
$$
see, for instance, \cite[p. 150]{BoMawhin} and \cite{Sc}. More precisely, as is well-known, the eigenvalues $\lambda_k$, for $k\in \ZZ[1,T]$, of problem \eqref{Do} are exactly the eigenvalues of the positive-definite matrix
$$
A:=\left(\begin{array}{ccccc}
  2 & -1 & 0 & ... & 0 \\
  -1 & 2 & -1 & ... & 0 \\
     &  & \ddots &  &  \\
  0 & ... & -1 & 2 & -1 \\
  0 & ... & 0 & -1 & 2
\end{array}\right)_{T\times T}.
$$
Thus it follows that
$$
\lambda_k=4\sin^2\left(\frac{k\pi}{2(T+1)}\right),\,\,\,\,\,\,\forall\, k\in \ZZ[1,T].
$$
}
\end{remark}
 A direct application of this result yields the following.

\begin{es}\label{esempio0}\rm{For every
$$
\alpha\in (\lambda_1,+\infty),
$$
the following second-order discrete problem
\begin{equation} \tag{$S_\alpha$} \label{dxi}
\left\{
\begin{array}{l}
-\Delta^2u({k-1})=\alpha\displaystyle\frac{u(k)}{1+u(k)^2},\quad \forall\; k  \in
{\mathbb{Z}}[1,T] \\
{u(0)=u({T+1})=0,}\\
\end{array}
\right.
\end{equation}
has at least one positive solution.
}
\end{es}

\begin{remark}\rm{
In Example \ref{esempio0}, for every $\alpha$ sufficiently large, our approach ensures the existence of at least one positive solution $u_\alpha\in H$ for the discrete problem \eqref{dxi}. A more delicate problem is to find a concrete expression of the function $u_\alpha$ that one may hope to be exploited by numerical methods.}
\end{remark}

\begin{remark}\rm{
We refer to the paper of Galewski and Orpel \cite{GO} for several multiplicity results on discrete partial difference equations. See also the papers \cite{A4,KMR,KMRT,MarcuMolica, MolicaRepovs1, MolicaRepovs2} for recent contributions to discrete problems.}
\end{remark}
{\bf Acknowledgements.}  This paper was written when the first author was visiting professor at the University of Ljubljana in 2013. He expresses his gratitude to the host institution for warm hospitality.
  The manuscript was realized within the auspices of the GNAMPA Project 2013 entitled {\it Problemi non-locali di tipo Laplaciano frazionario} and the SRA grants P1-0292-0101 and J1-5435-0101. The authors warmly thank the anonymous referees for their useful comments on the manuscript.

\end{document}